\documentclass[12pt, reqno]{amsart}

\newtheorem{theorem}{Theorem}[section]
\newtheorem{lemma}[theorem]{Lemma}
\newtheorem{proposition}[theorem]{Proposition}
\newtheorem{corollary}[theorem]{Corollary}   
\newtheorem{definition}[theorem]{Definition}
\newtheorem{example}[theorem]{Example}

\newtheorem{remark}[theorem]{Remark}

\newtheorem{question}[theorem]{Question}

\newtheorem{setup}{Setup}
\newtheorem{hypothesis}{Hypothesis}

\numberwithin{equation}{section}

\usepackage{times}
\usepackage{enumerate}
\usepackage{mathrsfs}
\usepackage{tfrupee}
\usepackage{tikz}
\usetikzlibrary{chains,fit}
\usetikzlibrary{shapes,snakes}
\usetikzlibrary{graphs}
\usepackage{graphicx,adjustbox}
\usepackage{hyperref}
\usepackage{amsmath,amssymb}
\usepackage{amscd}
\usepackage{graphicx}
\usepackage[all]{xy}

\usepackage{float}
\usepackage{caption}
\usepackage{subcaption}

\usepackage{cleveref}

\usepackage{geometry}
\newcommand{\ov}{\overline}

\begin{document}

\title{Cohen-Macaulay Property of Binomial Edge Ideals with Girth of Graphs}
\author{
Kamalesh Saha \and Indranath Sengupta
}
\date{}

%Author1
\address{\small \rm  Discipline of Mathematics, IIT Gandhinagar, Palaj, Gandhinagar, 
Gujarat 382355, INDIA.}
\email{kamalesh.saha@iitgn.ac.in}

%Author2
\address{\small \rm  Discipline of Mathematics, IIT Gandhinagar, Palaj, Gandhinagar, 
Gujarat 382355, INDIA.}
\email{indranathsg@iitgn.ac.in}
\thanks{The second author is the corresponding author; supported by the 
research grant CRG/2022/007047, sponsored by the SERB, Government of India.}

\date{}



\allowdisplaybreaks

\begin{abstract}
Conca and Varbaro (Invent. Math. 221 (2020), no. 3) showed the equality of depth of a graded ideal and its initial ideal in a polynomial ring when the initial ideal is square-free. In this paper, we give some beautiful applications of this fact in the study of Cohen-Macaulay binomial edge ideals. 
We prove that for the characterization of Cohen-Macaulay binomial edge ideals, it is enough to consider only ``biconnected graphs with some whisker attached" and this done by investigating the initial ideals. 
We give several necessary conditions for a binomial edge ideal to be Cohen-Macaulay in terms of 
smaller graphs. Also, under a hypothesis, we give a sufficient condition for Cohen-Macaulayness of 
binomial edge ideals in terms of blocks of graphs. Moreover, we show that a graph with Cohen-Macaulay 
binomial edge ideal has girth less than $5$ or equal to infinity.
\end{abstract}

\maketitle

\section{Introduction}
Corresponding to a graph, we can associate ideals of polynomial rings in several ways. 
This has become an attractive area of research in modern commutative algebra due to 
some rich applications of these ideals in other areas. One of the most studied ideals 
among them in the last decade has been the binomial edge ideals of graphs.
\medskip

Let $G$ be a simple graph on the vertex set $V(G)=[n]=\{1,\ldots,n\}$ with the edge set $E(G)$. For a fixed field $K$, consider the polynomial ring $\mathcal{S}=K[x_{1},\ldots,x_{n},y_{1},\ldots,y_{n}]$. Then the \textit{binomial edge ideal} of $G$, denoted by $J_{G}$, is an ideal of $\mathcal{S}$ defined as
$$J_{G}=\big<f_{ij} = x_{i}y_{j}-x_{j}y_{i}\mid \{i,j\}\in E(G)\,\, \text{with}\,\, i<j\big>.$$

The study of binomial edge ideals started in 2010 through the articles \cite{hhhrkara} 
and \cite{ohtani} independently. Binomial edge ideals can be seen as a generalization of determinantal ideals of $2$-minor of a $2\times n$ matrix of indeterminates. One of the primary motivations behind studying these ideals is its connection to Algebraic Statistics, particularly its appearance in the study of conditional independence statements \cite[Section 4]{hhhrkara}. Moreover, it is proved in \cite{cdg18} that binomial edge ideals belong to the class of famous \textit{Cartwright-Sturmfels} ideals, which was introduced in \cite{cdg20} inspired by the work of Cartwright and Sturmfels \cite{cs10} and has many nice properties.
\medskip

Generally, people study algebraic properties and invariants of binomial edge ideals by investigating combinatorics of the underlying graphs. One of the most important problems regarding these ideals is to classify Cohen-Macaulay binomial edge ideals combinatorially, as it will give the field independency of the Cohen-Macaulay property of binomial edge ideals. A substantial amount of research has been done in this direction (\cite{mont}, \cite{bms_cmbip}, \cite{acc}, \cite{ehh_cmbin}, \cite{hhhrkara}, \cite{s2acc}, \cite{ksara_cmunm}, \cite{cactus}, \cite{rin_smaldev}), but a complete combinatorial characterization of Cohen-Macaulay binomial edge ideals is still open. To give a combinatorial characterization of Cohen-Macaulay binomial edge ideals, Bolognini et al., in \cite{acc}, have introduced the notion of \textit{accessible} graphs (Definition \ref{defacc}). They proved that $J_{G}$ is Cohen-Macaulay implies $G$ is accessible and conjectured \cite[Conjecture 1.1]{acc} about the converse.
\medskip

In \cite[Corollary 2.7]{cv20}, Conca and Varbaro showed that for a graded ideal in a polynomial ring with a square-free initial ideal for some term order, the depth of that ideal and its initial ideal are equal. Using this fact and computing the depth of initial ideals of binomial edge ideals, we derive some results regarding the Cohen-Macaulay property of binomial edge ideals, some of which were proved for the case of accessibility in \cite{acc} and \cite{whisker}. For a given graph $G$, we show how beautifully the Cohen-Macaulay property of $J_{G}$ is related to the same of some other derived graphs from $G$, which appear naturally in the study of binomial edge ideals. Also, we show how the girth of a graph $G$ is connected with the Cohen-Macaulay property of $J_{G}$. The paper is written in the following fashion.
\medskip 

Section \ref{preli} is devoted to prerequisites of 
graph theory and commutative algebra for a better understanding of the rest of the paper.
\medskip 

In Section \ref{cminitial}, we first derive Lemma \ref{propcmcolon}, which is one of the key tools to prove the rest of the results in this section. By \cite[Corollary 2.7]{cv20}, it is clear that $J_{G}$ is Cohen-Macaulay if and only if $\mathrm{in}_{<}(J_{G})$ is Cohen-Macaulay, where $<$ is the lexicographic order on $\mathcal{S}$ induced by $x_{1}>\cdots>x_{n}>y_{1}>\cdots>y_{n}$. We prove in Lemma \ref{lemunmin} that $J_{G}$ is unmixed if and only if $\mathrm{in}_{<}(J_{G})$ is unmixed. By \cite[Proposition 5.13]{acc}, to prove \cite[Conjecture 1.1]{acc}, it is enough to show that every non-complete connected accessible graph $G$ has 
a cut vertex $v$ such that $G\setminus \{v\}$ is unmixed. Inspire by this fact, we prove the following:
\medskip

\noindent \textbf{Lemma \ref{lemG-vunm}, \ref{lembothnf}.} \textit{Let $G=G_{1}\cup G_{2}$ be a connected graph such that $V(G_{1})\cap V(G_{2})=\{v\}$ and $J_{G}$ be unmixed (resp. Cohen-Macaulay). If $v$ is not a free vertex of $G_{1}$ and $G_{2}$ both, then $J_{G\setminus\{v\}}$ is unmixed (resp. Cohen-Macaulay).}
\medskip

\noindent For a vertex $v$ of a graph $G$, we denote by $G_{v}$ the following graph:
$$ V(G_{v})=V(G)\,\, \text{and}\,\, E(G_{v})=E(G)\cup \{\{i,j\}\mid i,j\in \mathcal{N}_{G}(v), i\neq j\}.$$
In Proposition \ref{propGv}, we show that if $J_G$ is Cohen-Macaulay, then $J_{G_v}$ is also Cohen-Macaulay.

\noindent Let $G$ be a graph and $v$ be a non-free vertex of $G$. Then, 
by \cite[Lemma 4.8]{ohtani} and the proof of \cite[Theorem 1.1]{ehh_cmbin}, 
we have the following exact sequence:
$$ 0\longrightarrow \mathcal{S}/J_{G}\longrightarrow \mathcal{S}/J_{G_{v}}\oplus \mathcal{S}/\big<J_{G\setminus \{v\}},x_{v},y_{v}\big> \longrightarrow \mathcal{S}/\big<J_{G_{v}\setminus\{v\}}, x_{v},y_{v}\big>\longrightarrow 0.$$
The graphs $G\setminus \{v\}, G_{v}, G_{v}\setminus\{v\}$ play a crucial role in the study of binomial edge ideals, and these appear in many papers. In Proposition \ref{lemGiv} and \ref{propGv-v}, We try to give some necessary conditions for Cohen-Macaulayness of $J_{G}$ through the same property of some reduced graphs. As a corollary, we show how surprisingly the Cohen-Macaulay property of $J_{G}$ is related to the Cohen-Macaulay property of $J_{G\setminus\{v\}}, J_{G_v}$, and $J_{G_v\setminus\{v\}}$.
\medskip

\noindent \textbf{Corollary \ref{corbothnf}.} \textit{Let $G=G_{1}\cup G_{2}$ be a connected graph such that $V(G_{1})\cap V(G_{2})=\{v\}$ and $J_{G}$ be Cohen-Macaulay. If $v$ is a non-free vertex of $G_1$ and $G_2$ both, then $J_{G\setminus\{v\}}$, $J_{G_v}$ and $J_{G_v\setminus\{v\}}$ are Cohen-Macaulay. 
}
\medskip

\noindent Our main results give a sufficient and necessary condition for $J_G$ to be Cohen-Macaulay, where the sufficient condition depends on the following hypothesis.
\medskip

\noindent \textbf{Hypothesis \ref{hypo}.} \textit{Let $G$ be a graph and $v\in V(G)$ be a cut vertex of $G$. If $J_{G}$ is Cohen-Macaulay and $J_{G\setminus \{v\}}$ is unmixed, then $J_{G\setminus\{v\}}$ is Cohen-Macaulay.}
\medskip

We strongly believe that Hypothesis \ref{hypo} is true, but we are not able to prove it. If \cite[Conjecture 1.1]{acc} is true, then Hypothesis \ref{hypo} naturally holds. Also, due to Lemma \ref{lembothnf}, a partial case of the hypothesis holds true. Thus, we keep Hypothesis \ref{hypo} as an open problem. The main results of this section, which reduce the class of graphs to look at for the characterization of Cohen-Macaulay binomial edge ideals, are given below.
\medskip

\noindent\textbf{Theorem \ref{thmGcm}.} \textit{Let $G=G_{1}\cup G_{2}$ be a graph such that $V(G_{1})\cap V(G_{2})=\{v\}$. Consider the graph $\overline{G_{i}}$ by attaching a whisker to the graph $G_{i}$ at the vertex $v$ for $i=1,2$. If $J_{\ov{G_i}}$ is Cohen-Macaulay for each $i\in\{1,2\}$ and $J_{G}$ is unmixed, then $J_{G}$ is Cohen-Macaulay under the Hypothesis \ref{hypo}.
}
\medskip

\noindent\textbf{Theorem \ref{thmGbarcm}.} \textit{Let $G=G_{1}\cup G_{2}$ be a connected graph such that $V(G_{1})\cap V(G_{2})=\{v\}$. Consider the graph $\overline{G_{i}}$ by attaching a whisker to the graph $G_{i}$ at the vertex $v$ for $i=1,2$. If $J_{G}$ is Cohen-Macaulay, then $J_{\ov{G_{1}}}$ and $J_{\ov{G_{2}}}$ are Cohen-Macaulay.
}
\medskip

\noindent As an application of Theorem \ref{thmGcm} and \ref{thmGbarcm}, we get Corollary \ref{corblockcm1} and \ref{corblockcm2} respectively, given as an open problem in \cite[Question 5.11]{whisker}, which ensures that it is enough to focus only ``biconnected graphs with some whiskers attached" for the characterization of Cohen-Macaulay binomial edge ideals. Also, Corollary \ref{corblockcm1} and \ref{corblockcm2} boil down \cite[Conjecture 1.1]{acc} to ``biconnected graphs with some whiskers attached". 
For a graph $G$ with a large vertex and edge sets, it isn't easy to calculate the depth of $J_{G}$ and check Cohen-Macaulayness of $J_{G}$ using software algebra. Due to Corollary \ref{corblockcm1} and \ref{corblockcm2}, we can break a large graph $G$ into smaller pieces by looking at their blocks, and check the Cohen-Macaulayness of $J_G$ in an easier way. In the end, using Theorem \ref{thmGcm} and \ref{thmGbarcm}, we settle the open problem \cite[Problem 7.2]{acc} in Corollary \ref{coridentify}.
\medskip

An important invariant, which plays a crucial role in Graph Theory, is 
the girth of graphs. The \textit{girth} of a graph is defined as 
the length of a shortest 
induced cycle in $G$, denoted by $\mathrm{girth}(G)$. In Section \ref{secgirth}, we show how girth of graphs and Cohen-Macaulay property of binomial edge ideals are related to each other through the following theorem.
\medskip 

\noindent \textbf{Theorem \ref{thmgirth}.} \textit{Let $G$ be a graph. If $J_{G}$ is Cohen-Macaulay, then $\mathrm{girth}(G)\leq 4$ or $\mathrm{girth}(G)=\infty$.
}
\medskip

We end this article with a list of questions in Section \ref{secprob}.

\section{Preliminaries}\label{preli}
Let $R=K[x_{1},\ldots, x_{n}]$ be a polynomial ring over a field $K$, with the standard gradation. An Ideal $I$ of $R$ is said to be a \textit{monomial ideal} if $I$ is generated by a set of monomials and the unique minimal generating set of $I$ is denoted by $G(I)$. A monomial ideal $I$ is said to be \textit{square-free} if $G(I)$ consists of only square-free monomials. It is a well-known fact that square-free monomial ideals are radical ideals, and every associated prime of a square-free monomial ideal is generated by a subset of the set of variables in $R$.
\medskip

Let $I\subseteq R$ be a graded ideal. We denote the associated prime ideals of $I$ by $\mathrm{Ass}(I)$, the height of $I$ by $\mathrm{ht}(I)$, the depth of $R/I$ by $\mathrm{depth}(R/I)$, the Krull dimension of $R/I$ by $\mathrm{dim}(R/I)$, and the radical of $I$ by $\mathrm{rad}(I)$. Let $<$ be a term order on $R$. For a graded ideal ideal $I\subseteq R$, we denote the \textit{initial ideal} of $I$ with respect to $<$ by $\mathrm{in}_{<}(I)$. By saying an ideal $I$ of $R$ is Cohen-Macaulay, we mean the quotient ring $R/I$ is Cohen-Macaulay. We use the following results frequently in our proofs.

\begin{remark}\label{remdepth}{\rm
Let $I\subseteq R=K[\mathbf{x}]$, $J\subseteq R^{'}=K[\mathbf{y}]$ be two ideals and $S=K[\mathbf{x},\mathbf{y}]$, where $K$ is a field. Then $\mathrm{depth}(S/(I+J))=\mathrm{depth}(R/I)+\mathrm{depth}(R^{'}/J)$.
}
\end{remark}

\begin{remark}\label{remcmin}{\rm
Let $I\subseteq R$ be a graded ideal such that $\mathrm{in}_{<}(I)$ is square-free with respect to some term order. Then $\mathrm{depth}(R/I)=\mathrm{depth}(R/\mathrm{in}_{<}(I))$ by \cite[Corollary 2.7]{cv20}. Since $\mathrm{dim}(R/I)=\mathrm{dim}(R/\mathrm{in}_{<}(I))$, it follows that $R/I$ is Cohen-Macaulay if and only if $R/\mathrm{in}_{<}(I)$ is Cohen-Macaulay.
}
\end{remark}

\begin{lemma}[{\cite[Lemma 5.1]{dhs13}}]\label{lemheuneke}
Let $I\subseteq R$ be a square-free monomial ideal, and $\Delta$ be any subset of the variables in $R$. After relabelling the variables, we assume $\Delta=\{x_{1},\ldots,x_{k}\}$. Then either $\mathrm{depth}(R/I)=\mathrm{depth}(R/\big<I,x_{1},\ldots,x_{k}\big>)$ or there exists a $j\in \{1,\ldots,k\}$ such that $\mathrm{depth}(R/I)=\mathrm{depth}(R/(\big<I,x_{1},\ldots,x_{j-1}\big>:x_{j}))$. (We consider $x_{j}=0$ wherever applicable).
\end{lemma}

All graphs are assumed to be simple. For $T\subseteq V(G)$, we write $G\setminus T$ to denote the induced subgraph of $G$ on the vertex set $V(G)\setminus T$.
%If $\{u,v\}\in E(G)$, 
%then we say $u$ is adjacent to $v$ or vice versa. Similarly, we say $v$ is adjacent 
%to $A$ (or $A$ is adjacent to $v$) if $v$ is adjacent to a vertex in $A$, where 
%$A\subseteq V(G)$ and $v\in V(G)$.
%\medskip
For a vertex $v\in V(G)$, we say $\mathcal{N}_{G}(v)=\{u\in V(G)\mid \{u,v\}\in E(G)\}$ the \textit{neighbour set} of $v$ in $G$. We call $\vert \mathcal{N}_{G}(v)\vert$ 
the \textit{degree} of a vertex $v$ in $G$. If $\mathcal{N}_{G}(v)=\{u\}$, then $\{u,v\}\in E(G)$ 
is called a \textit{whisker} attached to $u$. A \textit{cycle} of length $n$ is a connected graph with $n$ vertices such that the degree of each vertex is $2$. A \textit{path} from $u$ to $v$ of length $n$ 
in $G$ is a sequence of vertices $u=v_{0},\ldots,v_{n}=v\in V(G)$, such that $\{v_{i-1},v_{i}\}\in E(G)$ for each $1\leq i\leq n$, and 
$v_{i}\neq v_{j}$ if $i\neq j$.

\begin{definition}{\rm
Let $G$ be a graph with $V(G)=[n]$. A path $\pi: i=i_{0},i_{1},\ldots,i_{r}=j$ from $i$ to $j$ with $i<j$ in $G$ is said to be an \textit{admissible path} if the following hold:
\begin{enumerate}
\item $i_{k}\neq i_{l}$ for $k\neq l$;

\item For each $k\in\{1,\ldots,r-1\}$, either $i_{k}<i$ or $i_{k}>j$;

\item The induced subgraph of $G$ on the vertex set $\{i_{0},\ldots,i_{r}\}$ has no induced cycle.
\end{enumerate}
}
\end{definition}

\begin{remark}{\rm
Corresponding to an admissible path $\pi: i=i_{0},i_{1},\ldots,i_{r}=j$ from $i$ to $j$ with $i<j$ in $G$, we associate the monomial
$$ u_{\pi}=\bigg(\prod_{i_{k}>j} x_{i_{k}}\bigg)\bigg(\prod_{i_{l}<i} y_{i_{l}}\bigg).$$
Then $\mathcal{G}=\{u_{\pi}f_{ij}\mid \pi\,\, \text{is an admissible path from}\,\, i\,\, \text{to}\,\, j\,\, \text{with}\,\, i<j\}$ is a reduced Gr\"{o}bner basis of $J_{G}$ with respect to $<$ by \cite[Theorem 2.1]{hhhrkara}. Therefore, 
$$G(\mathrm{in}_{<}(J_{G}))=\{u_{\pi}x_{i}y_{j}\mid \pi\,\, \text{is an admissible path from}\,\, i\,\, \text{to}\,\, j\,\, \text{with}\,\, i<j\}.$$
}
\end{remark}

A graph is said to be \textit{complete} if there is an edge between every pair of vertices, and we denote the complete graph on $n$ vertices by $K_{n}$. A vertex $v\in V(G)$ is called a \textit{free vertex} of $G$ if the induced subgraph of $G$ on $\mathcal{N}_{G}(v)$ is complete. A graph $G$ is said to be 
\textit{decomposable} into $G_{1}$ and $G_{2}$, if $G=G_{1}\cup G_{2}$, with $V(G_{1})\cap V(G_{2})=\{v\}$, such that $v$ is a free vertex of both $G_{1}$ and $G_{2}$.
\medskip

A vertex $v\in V(G)$ is said to be a \textit{cut vertex} of $G$, 
if removal of $v$ from $G$ increases the number of connected components. Let $G$ be a graph on the vertex set $V(G)=[n]$. A set $T\subseteq [n]$ is said to be a \textit{cutset} of $G$ if each $t\in T$ is a cut vertex of $G\setminus (T\setminus\{t\})$. We denote by $\mathscr{C}(G)$ the set of all cutsets of $G$. For $T\subseteq [n]$, we denote the number of connected components of the graph $G\setminus T$ by $c_{G}(T)$ (or sometimes by $c(T)$ if the graph is clearly understood from the context). Let $G_{1},\ldots,G_{c(T)}$ be the connected components of $G\setminus T$. For each $G_{i}$, we denote by $\tilde{G_{i}}$, the 
complete graph on the vertex set $V(G_{i})$. We set
$$ P_{T}(G)=\left\langle \bigcup_{i\in T}\lbrace x_{i},y_{i}\rbrace, J_{\tilde{G_{1}}},\ldots,J_{\tilde{G}_{c(T)}}\right\rangle.$$
Then $P_{T}(G)$ is a prime ideal. By \cite[Corollary 2.2]{hhhrkara}, $J_{G}$ is a radical ideal and from \cite[Corollary 3.9]{hhhrkara}, the minimal primary decomposition of $J_{G}$ is 
$$J_{G}=\bigcap_{T\in\mathscr{C}(G)} P_{T}(G).$$

From \cite[Lemma 3.1]{hhhrkara},  $\mathrm{ht}(P_{T}(G))=n+\vert T\vert -c(T)$. 
Since $\phi\in \mathscr{C}(G)$, $J_{G}$ is unmixed (i.e., heights of all minimal 
primes of $J_{G}$ are the same) if and only if $c(T)=\vert T\vert +c$, for every 
$T\in \mathscr{C}(G)$, where $c$ denotes the number of connected components of $G$. Therefore, if $J_{G}$ is unmixed, then $\mathrm{dim}(\mathcal{S}/J_{G})=n+c$. It follows from \cite[Proposition 2.1]{raufrin}, 
that, $v$ is a free vertex of $G$ if and only if $v\not\in T$ for all $T\in\mathscr{C}(G)$.

\begin{definition}[\cite{acc}, Definition 2.2]\label{defacc}{\rm
Let $G$ be a graph. A cutset $T\in \mathscr{C}(G)$ is said to be \textit{accessible} 
if there exists $t\in T$ such that $T\setminus\{t\}\in \mathscr{C}(G)$. The graph $G$ is said to be \textit{accessible} if $J_{G}$ is unmixed and every non-empty cutset $T\in\mathscr{C}(G)$ is accessible. 
}
\end{definition}

\begin{remark}{\rm
For a graph $G$ with connected components $G_{1},\ldots, G_{r}$, $J_{G}$ is 
unmixed (resp. Cohen-Macaulay) if and only if $J_{G_{i}}$ is 
unmixed (resp. Cohen-Macaulay) for each $i=1,\ldots,r$. So, we work with the blanket 
assumption that every graph is connected.
}
\end{remark}

\begin{remark}\label{remglu}{\rm
Let $G=G_{1}\cup G_{2}$, with $V(G_{1})\cap V(G_{2})=\{v\}$, be a decomposable graph, 
where $v$ is a free vertex of both $G_{1}$ and $G_{2}$. Then $J_{G}$ is unmixed (resp. Cohen-Macaulay) if and only if $J_{G_{i}}$ is 
unmixed (resp. Cohen-Macaulay) for each $i=1,2$ (see \cite{raufrin}).
}
\end{remark}

\section{Cohen-Macaulay Property of $J_{G}$ using Initial Ideal}\label{cminitial}
In this section, we study the Cohen-Macaulay property of binomial edge ideals in terms of 
some smaller graphs arising from the blocks of a given graph. We show that it is enough to 
investigate biconnected graphs attached with some whiskers for the classification of 
Cohen-Macaulay binomial edge ideals, if we assume 
a hypothesis. Due to Remark \ref{remcmin}, we prove our most of the results by looking into the initial ideals of binomial edge ideals. The following lemma is one of the key tools of this section.

\begin{lemma}\label{propcmcolon}
Let $I\subseteq R$ be a monomial ideal, and $f\in R$ be an arbitrary monomial. If $R/I$ is Cohen-Macaulay, then $R/(I:f)$ is Cohen-Macaulay with $\mathrm{depth}(R/(I:f))=\mathrm{depth}(R/I)$.
\end{lemma}

\begin{proof}
Let the degree of $f$ be $k$. Consider the following exact sequence:
$$0\longrightarrow R/(I:f)[-k] \overset{f}{\longrightarrow} R/I\longrightarrow R/\big<I,f\big>\longrightarrow 0.$$
We have $\mathrm{dim}(R/I)=\mathrm{max}\{\mathrm{dim}(R/(I:f)),\mathrm{dim}(R/\big<I,f\big>)\}$. Using \cite[Lemma 4.1]{chhktt19}, we get $\mathrm{depth}(R/(I:f))\geq \mathrm{depth}(R/I)=\mathrm{dim}(R/I)\geq \mathrm{dim}(R/(I:f))$, and hence the result follows.
\end{proof}

For a graph $G$ and $v\in V(G)$, it follows from \cite[Lemma 4.5]{acc} that $J_{G}$ is unmixed implies $J_{G_{v}}$ is unmixed and $G$ is accessible implies $G_{v}$ is accessible. The following Proposition \ref{propGv} is the extension of these results for the Cohen-Macaulay property of $J_{G}$.

\begin{proposition}\label{propGv}
Let $G$ be a graph and $v\in V(G)$. If $J_{G}$ is Cohen-Macaulay, then $J_{G_{v}}$ is Cohen-Macaulay.
\end{proposition}
\begin{proof}
Let $V(G)=\{1,\ldots,n\}$. Since labelling of vertices does not affect the properties of $\mathcal{S}/J_{G}$, we can choose a labelling of vertices of $G$ in such a way that $v=n$ and $\mathcal{N}_{G}(n)=\{n-1,\ldots, n-r\}$. By looking into the admissible paths in $G$, it is easy to observe that $(\mathrm{in}_{<}(J_{G}):x_{n})=\mathrm{in}_{<}(J_{G_{n}})$. Now, $J_{G}$ is Cohen-Macaulay implies $\mathrm{in}_{<}(J_{G})$ is Cohen-Macaulay by Remark \ref{remcmin} and therefore, we get $\mathrm{in}_{<}(J_{G_{n}})$ is Cohen-Macaulay using Lemma \ref{propcmcolon}. Hence, by Remark \ref{remcmin}, $J_{G_{n}}$ is Cohen-Macaulay.
\end{proof}

Let $G$ be a simple graph. Let $T\in \mathscr{C}(G)$ and $G_{1},\ldots,G_{c(T)}$ be the connected components of $G\setminus T$. For $\mathbf{v}=(v_{1},\ldots,v_{c(T)})\in V(G_{1})\times\cdots\times V(G_{c(T)})$, we consider the following prime ideal:
$$P_{T}(\mathbf{v})=\big<x_{i},y_{i}\mid i\in T\big>+ \sum_{k=1}^{c(T)}\big< \{x_{i},y_{j}\mid i,j\in V(G_{k}), i<v_{k}, j>v_{k}\}\big>.$$
By \cite[Lemma 1]{s2acc}, we get
$$\mathrm{in}_{<}(P_{T}(G))=\bigcap_{\mathbf{v}\in V(G_{1})\times\cdots\times V(G_{c(T)})} P_{T}(\mathbf{v}).$$
Since $\mathrm{in}_{<}(J_{G})$ is radical, by \cite[Corollary 1.12]{cdg18}, we have 
$$\mathrm{in}_{<}(J_{G})=\bigcap_{T\in\mathscr{C}(G)}  \mathrm{in}_{<}(P_{T}(G)).$$

\begin{proposition}
Let $G$ be a simple graph. Then
$$\mathrm{Ass}(\mathrm{in}_{<}(J_{G}))=\{P_{T}(\mathbf{v})\mid T\in \mathscr{C}(G)\,\, \text{and}\,\, \mathbf{v}\in V(G_{1})\times\cdots\times V(G_{c(T)})\}.$$
\end{proposition}

\begin{proof}
It is enough to prove that no $P_{T}(\mathbf{v})$ contains each other. Let $T,T^{'}$ be two different cutsets of $G$. Then there exists $i\in T$ such that $i\not\in T^{'}$ or there exists $i^{'}\in T^{'}$ such that $i^{'}\not\in T$. Without loss of generality, we assume there exists $i\in T$ such that $i\not\in T^{'}$. Then $x_{i},y_{i}\in P_{T}(\mathbf{v})$ for any $\mathbf{v}\in V(G_{1})\times\cdots\times V(G_{c(T)})$ but $x_{i},y_{i}$ both can not belong to any $P_{T^{'}}(\mathbf{v}^{'})$. Thus $P_{T}(\mathbf{v})\nsubseteq P_{T^{'}}(\mathbf{v}^{'})$. If there exists $i^{'}\in T^{'}$ such that $i^{'}\not\in T$, then $P_{T}(\mathbf{v}^{'})\not\subseteq P_{T}(\mathbf{v})$ and we are done. Let $T^{'}\subsetneq T$. Then one of connected components of $G\setminus T^{'}$, say $G_{k}^{'}$, contains $i$ as $i\not\in T^{'}$. Then $G\setminus T$ will have at least two connected components $G_{i_{1}}$ and $G_{i_{2}}$ such that $V(G_{i_{1}})\sqcup V(G_{i_{2}})\subseteq V(G_{k}^{'})\setminus\{i\}$. For any $P_{T}(\mathbf{v})$, there will be $p\in V(G_{i_{1}})$ and $q\in V(G_{i_{2}})$ such that $x_{p},y_{p},x_{q},y_{q}\not\in P_{T}(\mathbf{v})$. But, at least one of $x_{p},y_{p},x_{q},y_{q}$ belong to $P_{T^{'}}(\mathbf{v}^{'})$ as $p,q\in V(G_{k}^{'})$. Thus, $P_{T^{'}}(\mathbf{v}^{'})\nsubseteq P_{T}(\mathbf{v})$ for all $\mathbf{v}\in V(G_{1})\times\cdots\times V(G_{c(T)})$. Now, we will show $P_{T}(\mathbf{u})$ and $P_{T}(\mathbf{v})$ do not contain each other for any two different $\mathbf{u},\mathbf{v}\in V(G_{1})\times\cdots\times V(G_{c(T)})$. Since $\mathbf{u}$ and $\mathbf{v}$ are different, $u_{k}\neq v_{k}$ for some $k\in \{1,\ldots c(T)\}$. Then $x_{u_{k}},y_{u_{k}}\not\in P_{T}(\mathbf{u})$, but one of $x_{u_{k}}$ and $y_{u_{k}}$ belong to $P_{T}(\mathbf{v})$. Therefore, $P_{T}(\mathbf{v})\nsubseteq P_{T}(\mathbf{u})$. Similarly, $P_{T}(\mathbf{u})\nsubseteq P_{T}(\mathbf{v})$. Hence, for all $T\in\mathscr{C}(G)$ and all $\mathbf{v}\in V(G_{1})\times\cdots\times V(G_{c(T)})$, $P_{T}(\mathbf{v})$ is a minimal prime of $\mathrm{in}_{<}(J_{G})$ and the result follows.
\end{proof}

\begin{lemma}\label{lemunmin}
Let $G$ be a graph. Then $\mathrm{ht}(P_{T}(\mathbf{v}))=n+\vert T\vert -c_{G}(T)$. Moreover, $J_{G}$ is unmixed if and only if $\mathrm{in}_{<}(J_{G})$ is unmixed.
\end{lemma}
\begin{proof}
Consider any $P_{T}(\mathbf{v})\in \mathrm{Ass}(\mathrm{in}_{<}(J_{G}))$. From the construction of $P_{T}(\mathbf{v})$, we have
\begin{align*}
\mathrm{ht}(P_{T}(\mathbf{v}))&=\vert G(P_{T}(\mathbf{v}))\vert\\\
 &= 2 \vert T\vert +\sum_{k=1}^{c_{G}(T)}(\vert V(G_{k})\vert-1)\\
& =2\vert T\vert +n-\vert T\vert -c_{G}(T)\\
&= n+\vert T\vert -c_{G}(T).
\end{align*}
Since $\mathrm{ht}(P_{T}(G))=n+\vert T\vert -c_{G}(T)$ for all $T\in \mathscr{C}(G)$, $J_{G}$ is unmixed if and only if $\mathrm{in}_{<}(J_{G})$ is unmixed.
\end{proof}

Now, we are going to prove our main results and some of the results are dependent on the following hypothesis. 
\begin{hypothesis}\label{hypo}
    Let $G$ be a graph and $v\in V(G)$ be a cut vertex of $G$. If $J_{G}$ is Cohen-Macaulay and $J_{G\setminus \{v\}}$ is unmixed, then $J_{G\setminus\{v\}}$ is Cohen-Macaulay.
\end{hypothesis}

The above hypothesis is true for strongly unmixed binomial edge ideals by definition (see \cite[Definition 5.6]{acc}. Although, we strongly believe that Hypothesis \ref{hypo} is true in general, we are not able to fix it. So, we leave it as an open problem in Section \ref{secprob}. Note that if \cite[Conecjture 1.1]{acc} is affirmative, then Hypothesis \ref{hypo} naturally holds.

\begin{setup}\label{setup}{\rm
Let $G=G_{1}\cup G_{2}$ be a connected graph such that $V(G_{1})\cap V(G_{2})=\{v\}$. Consider the graph $\overline{G_{i}}$ by attaching a whisker to the graph $G_{i}$ at the vertex $v$ for $i=1,2$. Let $V(G)=\{1,\ldots,n\}$. Labelling of vertices does not affect the corresponding quotient ring of the binomial edge ideal of a graph. We relabel the vertices of $G$ such that 
\begin{enumerate}
\item[$\bullet$] $V(G_{1})=\{1,\ldots,m\}$ and $\mathcal{N}_{G_{1}}(m)=\{m-1,\ldots, m-r\}$;

\item[$\bullet$] $V(G_{2})=\{m,m+1,\ldots,n\}$ and $\mathcal{N}_{G_{2}}(m)=\{m+1,\ldots, m+s\}$.
\end{enumerate}

\noindent Note that here, $v=m$. We may assume 

\begin{enumerate}
\item[$\bullet$] $V(\ov{G_{1}})=\{1,\ldots,m+1\}$ and $E(\ov{G_{1}})=E(G_{1})\cup \{\{m,m+1\}\}$;

\item[$\bullet$] $V(\ov{G_{2}})=\{m-1,\ldots,n\}$ and $E(\ov{G_{2}})=E(G_{2})\cup \{\{m-1,m\}\}$.
\end{enumerate}

Let $I=\mathrm{in}_{<}(J_{G})$, $I_{i}=\mathrm{in}_{<}(J_{\ov{G_{i}}})$ and $J_{i}=\mathrm{in}_{<}(J_{G_{i}\setminus \{m\}})$ for  $i=1,2$. We write $\mathcal{S}=K[x_{j},y_{j}\mid j\in V(G)]$, $\mathcal{S}_{i}=K[x_{j},y_{j}\mid j\in V(\ov{G_{i}})]$, $A_{i}=K[x_{j},y_{j}\mid j\in V(G_{i}\setminus\{m\})]$, where $i=1,2$. By our choice of labelling, we have the following:
\begin{enumerate}
\item $G(I)=G(\mathrm{in}_{<}(J_{G_{1}}))\cup G(\mathrm{in}_{<}(J_{G_{2}}))$.
\item $\big<I_{1},x_{m}\big>=\big<I_{1}^{'},x_{m}\big>$, where $I_{1}^{'}\subseteq A_{1}[y_{m}]$ is an ideal.
\item $(I_{1}:x_{m})=\mathrm{in}_{<}(J_{(G_{1})_{m}}) + \big<y_{m+1}\big>$.
\item $(\big<I_{1},x_{m}\big>:y_{m})=I_{1}^{''}+\big<x_{m}\big>$, where $I_{1}^{''}$ is an ideal of $A_{1}$ containing $\{x_{m-1},\ldots,x_{m-r}\}$.
(Note that if $i=i_0,i_1,\ldots,i_r=j$ is an admissible path in $\ov{G_1}$, with length greater than $3$, then no $i_k<i$ can be $m$ and $i_r\neq m$ as $r>1$. Therefore, the only generators of $I_1$ as well as $\big<I_1,x_m\big>$, which are divisible by $y_m$, are $x_{m-1}y_m,\ldots,x_{m-r}y_{m}$ as $\mathcal{N}_{\ov{G_1}(m)}=\{m+1,m-1,\ldots,m-r\}$. Thus, $\{x_{m-1},\ldots,x_{m-r}\}\subseteq I^{''}$, where $I^{''}$ is a square-free monomial ideal which does not contain any generator divisible by $x_m$. Also, observe that $(\big<I_1,x_m\big>:y_m)=\big<(I_1:y_m),x_m\big>$ as no generator of $I_1$ is divisible by $x_my_m$. In a similar way, we can conclude the following.)
\item $\big<I_{1},x_{m},y_{m}\big>=J_{1}+\big<x_{m},y_{m}\big>$.\par 
\noindent (It is clear that $\big<I_1,x_m,y_m\big>=\mathrm{in}_{<}(J_{\ov{G_1}\setminus\{m\}})+\big<x_m,y_m\big>$. Now, $\ov{G_1}\setminus \{m\}$ has two connected components $G_1\setminus\{m\}$ and the vertex $m+1$. Thus, $G(\mathrm{in}_{<}(J_{\ov{G_1}\setminus\{m\}}))=G(\mathrm{in}_{<}(J_{G_1\setminus\{m\}}))=G(J_1)$, which gives the equality (5).)

\item $\big<I_{2},x_{m}\big>=\big<I_{2}^{'},x_{m}\big>$, where $I_{2}^{'}$ is an ideal of $A_{2}[x_{m-1},y_{m}]$ such that $G(I_{2}^{'})=G(I_{2})\setminus \{g\in G(I_{2})\mid x_{m}\mid g\}$.
\item $(I_{2}:x_{m})= \big<x_{m-1}y_{m}\big> + I_{2}^{''}$, where $I_{2}^{''}$ is an ideal of $A_{2}$ containing $\{y_{m+1},\ldots,y_{m+s}\}$.
\item $(\big<I_{2},x_{m}\big>:y_{m})=\mathrm{in}_{<}(J_{(G_{2})_{m}\setminus\{m\}}) + \big<x_{m-1},x_{m}\big>$.
\item $\big<I_{2},x_{m},y_{m}\big>=J_{2}+\big<x_{m},y_{m}\big>$.
\item $(I:y_m)=I_{1}^{''}+\mathrm{in}_{<}(J_{(G_2)_{m}})$.
\item $(I:x_m)=I_{2}^{''}+\mathrm{in}_{<}(J_{(G_1)_{m}})$.\par
\noindent (Let $\pi: i=i_0,i_1\ldots,i_r=j$ be an admissible path in $G$ such that $m\in \{i_1,\ldots,i_{r-1}\}$. Then either $\pi$ is an admissible path in $G_1$ or in $G_2$. Moreover, at least two neighbours of $m$ 
will belong to $\{i_0,\ldots,i_r\}$. Thus, any generator of $I$ divisible by $x_m$ and having 
degree greater than $2$ will be divisible by one of $\{x_{m-1},\ldots,x_{m-r}\}$. Hence, the equality (10) follows and similarly, (11) follows.)
\end{enumerate}
}
\end{setup}

\begin{lemma}\label{lemG-vunm}
Let $G=G_{1}\cup G_{2}$ be a connected graph such that $V(G_{1})\cap V(G_{2})=\{v\}$ and $J_{G}$ be unmixed. If $v$ is not a free vertex of $G_{1}$ and $G_{2}$ both, then $J_{G\setminus\{v\}}$ is unmixed. 
\end{lemma}

\begin{proof}
Since $J_{G}$ is unmixed, $G\setminus\{v\}$ has two connected components $H_{1}$ and $H_{2}$. Suppose $J_{G\setminus\{v\}}$ is not unmixed. Then by \cite[Proposition 5.2]{acc}, there exists either a cutset of $H_{1}$ containing $\mathcal{N}_{H_{1}}(v)$ or a cutset of $H_{2}$ containing $\mathcal{N}_{H_{2}}(v)$. Without loss of generality, we assume there exists $T_{1}\in\mathscr{C}(H_{1})$ such that $\mathcal{N}_{H_{1}}(v)\subseteq T_{1}$. Note that $\mathcal{N}_{H_{1}}(v)\subseteq T_{1}\in \mathscr{C}(H_{1})$ implies $T_{1}\in \mathscr{C}(G)$ and hence, $c_{G}(T_{1})=\vert T_{1}\vert +1$. One of the connected components of $G\setminus T_{1}$ is $G_{2}$. Since, $v$ is not a free vertex of $G_{2}$, there exists $T_{2}\in\mathscr{C}(G_{2})$ such that $v\in T_{2}$ by \cite{raufrin}. Then by \cite[Proposition 3.1]{whisker}, $T_{2}\in \mathscr{C}(G)$ and $J_{G}$ being unmixed $c_{G}(T_{2})=\vert T_{2}\vert+1$. One of the connected components of $G\setminus T_{2}$ is $H_{2}$. Again, by \cite[Proposition 3.1]{whisker}, we have $T=T_{1}\sqcup T_{2}\in \mathscr{C}(G)$. It is easy to observe that $c_{G}(T)=\vert T_{1}\vert +\vert T_{2}\vert $, which is a contradiction to the fact that $J_{G}$ is unmixed. Therefore, our assumption was wrong, and $J_{G\setminus\{v\}}$ is unmixed.
\end{proof}

\begin{lemma}\label{lembothnf}
Let $G=G_{1}\cup G_{2}$ be a connected graph such that $V(G_{1})\cap V(G_{2})=\{v\}$ and $J_{G}$ be Cohen-Macaulay. If $v$ is not a free vertex of $G_{1}$ and $G_{2}$ both, then $J_{G\setminus\{v\}}$ is Cohen-Macaulay.
\end{lemma}

\begin{proof}
We consider the labelling of $G$ as described in Setup \ref{setup} and in this setup, $v=m$. By the given condition, $m$ is not a free vertex of $G_2$ both. Then there exists $m+i,m+j\in \mathcal{N}_{G_2}(m)$ with $i<j$ such that $\{m+i,m+j\}\not\in E(G_2)$. Thus, $m+i,m,m+j$ is an admissible path in $G$ and so, $y_mx_{m+i}y_{m+j}\in G(I)$. Also, $\{m,m+i\}\in E(G)$ implies $x_my_{m+i}\in G(I)$. Now, consider the ideal $(I:x_{m+i}y_{m+j})$ and note that $(I:x_{m+i}y_{m+j})=J_1+\big<x_m,y_m\big>+J_{2}^{'}$, where $J_{2}^{'}$ is an ideal of $A_2$. Since $\mathcal{S}/I$ is Cohen-Macaulay, $\mathcal{S}/(I:x_{m+i}y_{m+j})$ is Cohen-Macaulay by Lemma \ref{propcmcolon}. Hence, $A_1/J_1$ is Cohen-Macaulay. Since $m$ is not a free vertex of $G_1$, in a similar way, it follows that $A_2/J_2$ is Cohen-Macaulay. Since $J_1+J_2=\mathrm{in}_{<}(J_{G\setminus \{v\}})$, by Remark \ref{remcmin}, $J_{G\setminus \{v\}}$ is Cohen-Macaulay.
\end{proof}

\begin{lemma}\label{lemonenf}
Let $G=G_{1}\cup G_{2}$ be a connected graph such that $V(G_{1})\cap V(G_{2})=\{v\}$ and $J_{G}$ be Cohen-Macaulay. Consider the labelling of $G$ as described in Setup \ref{setup} with $v=m$. If $m$ is a non-free vertex of $G_{1}$ and a free vertex of $G_{2}$, then $I_{1}^{'}$ and $J_{G_{2}\setminus\{m\}}$ is Cohen-Macaulay.
\end{lemma}

\begin{proof}
 Since $m$ is not a free vertex of $G_{1}$, there exists $m-i,m-j\in \mathcal{N}_{G_1}(m)$ with $i>j$ such that $\{m-i,m-j\}\not\in E(G_1)$. Thus, $m-i,m,m-j$ is an admissible path in $G$ and so, $x_mx_{m-i}y_{m-j}\in G(I)$. Also, $\{m-i,m\}\in E(G)$ implies $x_{m-i}y_{m}\in G(I)$. Now, consider the ideal $(I:x_{m-i}y_{m-j})$ and note that $(I:x_{m-i}y_{m-j})=J_{1}^{'}+\big<x_m,y_m\big>+J_{2}$, where $J_{1}^{'}$ is an ideal of $A_1$. Since $\mathcal{S}/I$ is Cohen-Macaulay, $\mathcal{S}/(I:x_{m+i}y_{m+j})$ is Cohen-Macaulay by Lemma \ref{propcmcolon}. Hence, $A_2/J_2$ is Cohen-Macaulay and this gives $J_{G_{2}\setminus\{m\}}$ is Cohen-Macaulay by Remark \ref{remcmin}. Again, consider the ideal $(I:y_{m+1})$. Note that $\big<I_1,x_m\big>=\big<\mathrm{in}_{<}(J_{G_1}),x_my_{m+1},x_m\big>=\big<I_{1}^{'},x_m\big>$. Thus, we can write $(I:y_{m+1})=I_{1}^{'}+\big<x_m\big>+J_{2}^{'}$, where $J_{2}^{'}$ is an ideal of $A_2$. Now, $\mathcal{S}/I$ is Cohen-Macaulay implies $\mathcal{S}/(I:y_{m+1})$ is Cohen-Macaulay by Lemma \ref{propcmcolon} Thus, $A_{1}[y_m]/I_{1}^{'}$ is Cohen-Macaulay. 
\end{proof}

\begin{proposition}\label{lemGiv}
Let $G=G_{1}\cup G_{2}$ be a connected graph such that $V(G_{1})\cap V(G_{2})=\{v\}$ and $J_{G}$ be Cohen-Macaulay. Then $J_{(G_1)_{v}}$ and $J_{(G_2)_{v}}$ are Cohen-Macaulay.
\end{proposition}

\begin{proof}
    Consider the labelling of $G$ as described in Setup \ref{setup} with $v=m$. We consider the ideal $I=\mathrm{in}_{<}(J_G)$. Note that the only degree two generators of $I$ divisible by $x_m$ are $x_my_{m+1},\ldots,x_{m}y_{m+s}$. Again, if $u\in G(I)$ is a generator of degree greater than $2$ and divisible by $x_m$, then $u\in G(\mathrm{in}_{<}(J_{G_1}))$. By looking into the admissible paths in $G_1$, it is easy to verify that $(\mathrm{in}_{<}(J_{G_1}):x_{m})=\mathrm{in}_{<}(J_{(G_{1})_{m}})$. Therefore, we can write $(I:x_m)=\mathrm{in}_{<}(J_{(G_{1})_{m}})+I_{2}^{''}$. Now, $J_G$ s Cohen-Macaulay implies $\mathcal{S}/I$ is Cohen-Macaulay by Remark \ref{remcmin}. Thus, by Lemma \ref{propcmcolon}, $\mathcal{S}/(I:x_m)$ is Cohen-Macaulay. Hence, $J_{(G_1)_m}$ is Cohen-Macaulay by Remark \ref{remcmin}. Similarly, considering the ideal $(I:y_m)$, we get $J_{(G_2)_m}$ is Cohen-Macaulay.
\end{proof}

\begin{proposition}\label{propGv-v}
    Let $G$ be a graph such that $J_{G}$ is Cohen-Macaulay. If there is a cut vertex $v$ of $G$ for which $J_{G\setminus\{v\}}$ is Cohen-Macaulay, then $J_{G_{v}\setminus \{v\}}$ is Cohen-Macaulay.
\end{proposition}

\begin{proof}
    Since $v$ is a non-free vertex of $G$, due to \cite[Lemma 4.8]{ohtani}, we get the following exact sequence:
    $$ 0\longrightarrow \mathcal{S}/J_{G}\longrightarrow \mathcal{S}/J_{G_{v}}\oplus \mathcal{S}/\big<J_{G\setminus \{v\}},x_{v},y_{v}\big> \longrightarrow \mathcal{S}/\big<J_{G_{v}\setminus\{v\}}, x_{v},y_{v}\big>\longrightarrow 0.$$
Without loss of generality, we may assume $G$ is connected. Then $\mathrm{depth}(\mathcal{S}/J_{G})=n+1$. Since $v$ is a cut vertex of $G$ and $J_G$ is unmixed, $G\setminus\{v\}$ has two connected components. Thus, $J_{G\setminus \{v\}}$ is Cohen-Macaulay implies $\mathrm{depth}(\mathcal{S}/\big<J_{G\setminus\{v\}},x_v,y_v\big>)=n-1+2=n+1$. Again, by Proposition \ref{propGv}, $J_{G_v}$ is Cohen-Macaulay and hence, $\mathrm{depth}(\mathcal{S}/J_{G_v})=n+1$. Therefore, using \cite[Proposition 1.2.9]{cmring}, we get $\mathrm{depth}(\mathcal{S}/\big<J_{G_v\setminus\{v\}},x_v,y_v\big>)\geq n$. Now, $J_G$ is Cohen-Macaulay and $J_{G\setminus\{v\}}$ is unmixed together imply $J_{G_v\setminus \{v\}}$ is unmixed by \cite[Corollary 5.16]{acc}. Therefore, $\mathrm{dim}(\mathcal{S}/\big<J_{G_v\setminus\{v\}},x_v,y_v\big>)=n$ as $G_{v}\setminus\{v\}$ is connected and hence, $J_{G_v\setminus\{v\}}$ is Cohen-Macaulay.
\end{proof}

\begin{corollary}\label{corbothnf}
    Let $G=G_{1}\cup G_{2}$ be a connected graph such that $V(G_{1})\cap V(G_{2})=\{v\}$ and $J_{G}$ be Cohen-Macaulay. If $v$ is a non-free vertex of $G_1$ and $G_2$ both, then $J_{G\setminus\{v\}}$, $J_{G_v}$ and $J_{G_v\setminus\{v\}}$ are Cohen-Macaulay. 
\end{corollary}

\begin{proof}
    By Proposition \ref{propGv}, Lemma \ref{lembothnf}, and Proposition \ref{propGv-v}, the result follows.
\end{proof}

 Due to Remark \ref{remglu}, Proposition \ref{propGv}, Lemma \ref{lembothnf}, and Proposition \ref{propGv-v}, we remark the following.
 
\begin{remark}\label{thmG-vcm}
Let $G$ be a graph such that $J_{G}$ is Cohen-Macaulay. If there is a cut vertex $v$ of $G$ for which $J_{G\setminus\{v\}}$ is unmixed, then $J_{G\setminus\{v\}}, J_{G_{v}}, J_{G_{v}\setminus \{v\}}$ are Cohen-Macaulay assuming Hypothesis \ref{hypo}. Moreover, if $v\in V(G)$ is a free vertex of $G$, then $J_{G}$ is Cohen-Macaulay and $J_{G\setminus\{v\}}$ is unmixed together imply $J_{G\setminus\{v\}}$ is Cohen-Macaulay.
\end{remark}

\begin{lemma}\label{lemGv-vunm}
    Let $G$ be a graph and $v\in V(G)$ is not a cut vertex of $G$. If $J_{G_v}$ is unmixed and $J_{G\setminus \{v\}}$ is unmixed, then $J_{{G_v}\setminus\{v\}}$ is unmixed.
\end{lemma}
\begin{proof}
    Without loss of generality, assume $G$ is connected. Let $T\in \mathscr{C}(G_{v}\setminus\{v\})$. Suppose $\mathcal{N}_{G}(v)\subseteq T$. Note that $T$ is a cutset of $G\setminus \{v\}$ also and $c_{G\setminus\{v\}}(T)=c_{G_{v}\setminus\{v\}}(T)$ as $v$ is not a cut vertex of $G$. Since $J_{G\setminus\{v\}}$ is unmixed and $G\setminus\{v\}$ is connected (as $v$ is not a cut vertex of $G$), $c_{G\setminus\{v\}}(T)=c_{G_{v}\setminus\{v\}}(T)=\vert T\vert+1$. Again, it is easy to observe that $T\in\mathscr{C}(G_v)$ and $c_{G_v}(T)=\vert T\vert+2$. This is a contradiction to the fact that $J_{G_v}$ is unmixed. Thus, there is no cutset of $G_{v}\setminus\{v\}$ containing $\mathcal{N}_{G}(v)$. Hence, by \cite[Corollary 4.3]{ccm}, we have $J_{G_v\setminus\{v\}}$ is unmixed.
\end{proof}
 
\begin{theorem}\label{thmGcm}
Let $G=G_{1}\cup G_{2}$ be a graph such that $V(G_{1})\cap V(G_{2})=\{v\}$. Consider the graph $\overline{G_{i}}$ by attaching a whisker to the graph $G_{i}$ at the vertex $v$ for $i=1,2$. If $J_{\ov{G_i}}$ is Cohen-Macaulay for each $i\in\{1,2\}$ and $J_{G}$ is unmixed, then $J_{G}$ is Cohen-Macaulay under the Hypothesis \ref{hypo}.
\end{theorem}

\begin{proof}
If $v$ is a free vertex of both $G_{1}$ and $G_{2}$, then without assuming Hypothesis \ref{hypo}, the result follows by Remark \ref{remglu}. So, we may assume $G$ is not decomposable. Let $G$ be connected with $V(G)=\{1,\ldots,n\}$. We label the vertices of $G$ as described in Setup \ref{setup} and in this case, $v=m$. Since $J_{G}$ is unmixed and $\{m\}\in \mathscr{C}(G)$, $c_{G}(\{m\})=2$ and thus, $G_{i}\setminus\{m\}$ is connected for each $i\in\{1,2\}$. Now, $\mathcal{S}_{i}/J_{\ov{G_{i}}}$ is Cohen-Macaulay implies $\mathcal{S}_{i}/I_{i}$ is Cohen-Macaulay for each $i\in\{1,2\}$ by Remark \ref{remcmin}. Moreover, $\mathrm{depth}(\mathcal{S}_{1}/I_{1})=\mathrm{dim}(\mathcal{S}_{1}/I_{1})=m+2, \mathrm{depth}(\mathcal{S}_{2}/I_{2})=\mathrm{dim}(\mathcal{S}_{2}/I_{2})=n-m+3$. Now, we proceed by considering the following possible cases.
\medskip

\noindent \textbf{Case I.} Let $m$ be a non-free vertex in $G_{1}$ and $G_{2}$ both. Then by Lemma \ref{lemG-vunm}, $J_{G\setminus\{m\}}$ is unmixed. Therefore, $J_{{\ov{G_1}}\setminus\{m\}}$ and $J_{\ov{{G_2}}\setminus\{m\}}$ are both unmixed. Since $J_{\ov{G_1}}$ and $J_{\ov{G_2}}$ are both Cohen-Macaulay, by Hypothesis \ref{hypo}, $J_{\ov{G_i}\setminus \{m\}}$ is Cohen-Macaulay for each $i\in\{1,2\}$. Hence, $J_1$ and $J_2$ are Cohen-Macaulay by Remark \ref{remcmin} with 
\begin{equation}\label{eq3.1}
\mathrm{depth}(A_{1}/J_{1})=m\,\,\text{and}\,\,\mathrm{depth}(A_{2}/J_{2})=n-m+1
\end{equation}
Now let us focus on the ideal $I_{1}$. Since $\mathcal{S}_{1}/I_{1}$ is Cohen-Macaulay, by Lemma \ref{propcmcolon}, $\mathcal{S}_{1}/(I_{1}:y_{m})$ is Cohen-Macaulay with depth and dimension equal to $m+2$. We can write $(I_{1}:y_{m})=\big<x_{m}y_{m+1}\big>+ I_{1}^{''}$, where $I_{1}^{''}\subseteq A_{1}$. It is clear that $K[x_{m},y_{m+1}]/\big<x_{m}y_{m+1}\big>$ is Cohen-Macaulay with dimension $1$ and therefore, $A_{1}/I_{1}^{''}$ is Cohen-Macaulay with
\begin{equation}\label{eq3.2}
\mathrm{depth}(A_{1}/I_{1}^{''})=m-1.
\end{equation}
In a similar way, $\mathcal{S}_{2}/(I_{2}:x_{m})$ is Cohen-Macaulay with depth and dimension equal to $n-m+3$. We can write $(I_{2}:x_{m})=\big<x_{m-1}y_{m}\big>+ I_{2}^{''}$, where $I_{2}^{''}\subseteq A_{2}$. Therefore, $A_{2}/I_{2}^{''}$ is Cohen-Macaulay with
\begin{equation}\label{eq3.3}
\mathrm{depth}(A_{2}/I_{2}^{''})=n-m.
\end{equation}
Again, due to Proposition \ref{lemGiv}, $J_{(G_1)_m}$ is Cohen-Macaulay. Hence, $A_{1}[x_{m},y_{m}]/\mathrm{in}_{<}(J_{(G_{1})_{m}})$ is Cohen-Macaulay by Remark \ref{remcmin} and
\begin{equation}\label{eq3.4}
\mathrm{depth}(A_{1}[x_{m},y_{m}]/\mathrm{in}_{<}(J_{(G_{1})_{m}}))=m+1.
\end{equation}
Similarly, $J_{(G_2)_m}$ is Cohen-Macaulay. By Lemma \ref{lemG-vunm}, $J_{G\setminus \{m\}}$ is unmixed and so, $J_{G_{2}\setminus\{m\}}$ is unmixed. Since $J_{(G_2)_m}$ is unmixed, by Lemma \ref{lemGv-vunm}, $J_{(G_{2})_{m}\setminus\{m\}}$ is unmixed. Consider the graph $\ov{(G_2)_m}$ by attaching a whisker at $m$ to the graph $(G_2)_m$. By Remark \ref{remglu}, ${J_{\ov{(G_2)_m}}}$ is Cohen-Macaulay as $J_{(G_2)_m}$ is so. Hence, by Hypothesis \ref{hypo}, $J_{(G_{2})_{m}\setminus\{m\}}$ is Cohen-Macaulay. Thus, by Remark \ref{remcmin},
 \begin{equation}\label{eq3.5}
\mathrm{depth}(A_{2}/\mathrm{in}_{<}(J_{(G_{2})_{m}\setminus\{m\}}))=n-m+1.
\end{equation}
Now, we consider $I=\mathrm{in}_{<}(J_{G})$ to prove $J_{G}$ is Cohen-Macaulay. By our given labelling, we get the following 
\begin{align*}
(I:x_{m})&=\mathrm{in}_{<}(J_{(G_{1})_{m}})+I_{2}^{''};\\
(\big<I,x_{m}\big>:y_{m})&=I_{1}^{''}+\mathrm{in}_{<}(J_{(G_{2})_{m}\setminus\{m\}})+\big<x_{m}\big>;\\
\big<I,x_{m},y_{m}\big>&=J_{1} + J_{2}+\big<x_{m},y_{m}\big>.
\end{align*} 

\noindent Using equations (\ref{eq3.1}), (\ref{eq3.2}), (\ref{eq3.3}), (\ref{eq3.4}), (\ref{eq3.5}) and Remark \ref{remdepth}, we get
\begin{align*}
\mathrm{depth}(\mathcal{S}/(I:x_{m}))&=\mathrm{depth}(A_{1}[x_{m},y_{m}]/\mathrm{in}_{<}(J_{(G_{1})_{m}}))+\mathrm{depth}(A_{2}/I_{2}^{''})\\ &=m+1+n-m=n+1;\\
\mathrm{depth}(\mathcal{S}/(\big<I,x_{m}\big>:y_{m}))&= \mathrm{depth}(A_{1}/I_{1}^{''})+\mathrm{depth}(A_{2}/\mathrm{in}_{<}(J_{(G_{2})_{m}\setminus\{m\}}))+1\\ &= m-1+n-m+1+1\\ &=n+1;\\
\mathrm{depth}(\mathcal{S}/\big<I,x_{m},y_{m}\big>)&= \mathrm{depth}(A_{1}/J_{1})+ \mathrm{depth}(A_{2}/J_{2})\\
&=m+n-m+1=n+1.
\end{align*}

\noindent By Lemma \ref{lemheuneke}, we have $\mathrm{depth}(\mathcal{S}/I)=n+1=\mathrm{dim}(\mathcal{S}/I)$, which gives $\mathcal{S}/I$ is Cohen-Macaulay. Hence $J_{G}$ is Cohen-Macaulay by Remark \ref{remcmin}.
\medskip

\noindent \textbf{Case II.} Let $m$ be a free vertex of exactly one of $G_1$ and $G_2$. Without loss of generality, we assume $m$ is a free vertex of $G_{2}$ and not a free vertex of $G_{1}$. We claim that $J_{G_2\setminus \{m\}}$ is unmixed. Let $T\in \mathscr{C}(G_2\setminus \{m\})$ such that $\mathcal{N}_{G_2}(m)\subseteq T$. Then $T\in \mathscr{C}(\ov{G_2})$ and $c_{\ov{G_2}}(T)=\vert T\vert +1$ as $J_{\ov{G_2}}$ is unmixed. Since $v$ is a non-free vertex of $G_1$, we can choose a cutset $S$ of $G_1$  such that $v\in S$. Note that $S$ is also a cutset of $\ov{G_1}$ and $c_{\ov{G_1}}(S)=\vert S\vert +1$ as $J_{\ov{G_1}}$ is unmixed. Now, one connected component of $\ov{G_2}\setminus T$ is the edge $\{m-1,m\}$ and one connected component of $\ov{G_1}\setminus S$ is the vertex $m+1$. By \cite[Proposition 3.1]{whisker}, $T\cup S$ is a cutset of $G$ and it is clear that $c_{G}(T\cup S)=\vert T\vert+\vert S\vert$, which is a contradiction to the fact that $J_G$ is unmixed. Thus, no cutset of $G_2\setminus \{m\}$ contain $\mathcal{N}_{G_2}(m)$. Therefore, by \cite[Proposition 5.2]{acc}, $J_{\ov{{G_2}}\setminus\{m\}}$ is unmixed and hence, $J_{G_2\setminus\{m\}}$ is Cohen-Macaulay by Hypothesis \ref{hypo}. By Remark \ref{remcmin}, $A_{2}/J_2$ is Cohen-Macaulay with 
\begin{equation}\label{eq3.6}
    \mathrm{depth}(A_{2}/J_{2})=n-m+1.
\end{equation}
Again, $\mathcal{S}_{2}/(I_{2}:x_{m})$ is Cohen-Macaulay with $\mathrm{depth}(\mathcal{S}_{2}/(I_{2}:x_{m}))=n-m+3$ by Lemma \ref{propcmcolon}. We can write $(I_{2}:x_{m})=\big<x_{m-1}y_{m}\big>+ I_{2}^{''}$, where $I_{2}^{''}\subseteq A_{2}$ and so, 
\begin{equation}\label{eq3.7}
\mathrm{depth}(A_{2}/I_{2}^{''})=n-m.
\end{equation}
Since $J_{\ov{G_1}}$ is Cohen-Macaulay, by Proposition \ref{lemGiv}, it follows that $J_{(G_1)_m}$ is Cohen-Macaulay and thus, by Remark \ref{remcmin}, we get 
\begin{equation}\label{eq3.8}
\mathrm{depth}(A_{1}[x_m,y_m]/\mathrm{in}_{<}(J_{(G_1)_m}))=m+1.
\end{equation}
Now, let us focus on the ideal $I_{1}^{'}\subseteq A_1[y_m]$, where $G(I_{1}^{'})=G(I_{1})\setminus \{g\in G(I_{1}): x_{m}\mid g\}$. Then $\mathfrak{p}\in \mathrm{Ass}(I_{1}^{'})$ implies $\big<\mathfrak{p},x_{m}\big>\in \mathrm{Ass}(I_{1})$ as $x_my_{m+1}\in I_1$. Therefore, $I_{1}$ being unmixed, $\mathrm{ht}(I_{1}^{'})=\mathrm{ht}(I_{1})-1=m-1$ and so, $\mathrm{dim}(A_{1}[y_{m}]/I_{1}^{'})=(2m-1)-(m-1)=m$. Hence, by Lemma \ref{lemonenf}, $A_1[y_m]/I_{1}^{'}$ is Cohen-Macaulay with
\begin{equation}\label{eq3.9}
    \mathrm{depth}(A_{1}[y_m]/I_{1}^{'})=m.
\end{equation}

\noindent Now, due to our given labelling by looking at the admissible paths in $G$, it is easy to verify that the following hold in this case. 
\begin{align*}
I:x_{m}&=\mathrm{in}_{<}(J_{(G_{1})_{m}})+I_{2}^{''};\\
\big<I,x_{m}\big>&=I_{1}^{'}+J_{2}+\big<x_{m}\big>.
\end{align*} 
\noindent Using equations (\ref{eq3.6}), (\ref{eq3.7}), (\ref{eq3.8}), (\ref{eq3.9}), and Remark \ref{remdepth}, we get
\begin{align*}
\mathrm{depth}(\mathcal{S}/(I:x_{m}))&=\mathrm{depth}(A_{1}[x_{m},y_{m}]/\mathrm{in}_{<}(J_{(G_{1})_{m}}))+\mathrm{depth}(A_{2}/I_{2}^{''})\\ &=m+1+n-m=n+1;\\
\mathrm{depth}(\mathcal{S}/\big<I,x_{m}\big>)&= \mathrm{depth}(A_{1}[y_{m}]/I_{1}^{'})+\mathrm{depth}(A_{2}/J_{2})\\ &= m+n-m+1\\ &=n+1.
\end{align*}

\noindent Thus, $\mathrm{depth}(\mathcal{S}/I)=n+1=\mathrm{dim}(\mathcal{S}/I)$ by Lemma \ref{lemheuneke} and so, $\mathcal{S}/I$ is Cohen-Macaulay. Hence $J_{G}$ is Cohen-Macaulay by Remark \ref{remcmin}.
\end{proof}

Assuming Hypothesis \ref{hypo}, we proved Theorem \ref{thmGcm}. Now, we  prove the converse of 
Theorem \ref{thmGcm} and we do not need to assume Hypothesis \ref{hypo} in this proof.

\begin{theorem}\label{thmGbarcm}
Let $G=G_{1}\cup G_{2}$ be a connected graph such that $V(G_{1})\cap V(G_{2})=\{v\}$. Consider the graph $\overline{G_{i}}$ by attaching a whisker to the graph $G_{i}$ at the vertex $v$ for $i=1,2$. If $J_{G}$ is Cohen-Macaulay, then $J_{\ov{G_{1}}}$ and $J_{\ov{G_{2}}}$ are Cohen-Macaulay.
\end{theorem}

\begin{proof}
We label the vertices of $G$ as described in Setup \ref{setup} and in this case, $v=m$. We will consider two possible cases:
\medskip

\noindent \textbf{Case I.} $m$ is not a free vertex of both $G_{1}$ and $G_{2}$. Then by Lemma \ref{lembothnf}, $J_{G_{i}\setminus\{m\}}$ is Cohen-Macaulay and hence, by Remark \ref{remcmin}, $A_{i}/J_{i}$ is Cohen-Macaulay for each $i\in\{1,2\}$ with
\begin{equation}\label{eq3.10}
    \mathrm{depth}(A_1/J_1)=m\,\,\text{and}\,\,\mathrm{depth}(A_2/J_2)=n-m+1.
\end{equation}
Now, $J_G$ is Cohen-Macaulay implies $\mathcal{S}/I$ is Cohen-Macaulay by Remark \ref{remcmin}. Therefore, by Lemma \ref{propcmcolon}, $\mathcal{S}/(I:x_m)$ and $\mathcal{S}/(I:y_m)$ are Cohen-Macaulay. From our labelling of vertices of $G$, we have
\begin{align*}
(I:x_{m})&=\mathrm{in}_{<}(J_{(G_{1})_{m}})+I_{2}^{''};\\
(I:y_{m})&=I_{1}^{''}+\mathrm{in}_{<}(J_{(G_{2})_{m}}).
\end{align*}
Thus, $A_1[x_m,y_m]/\mathrm{in}_{<}(J_{(G_1)_m})$ is Cohen-Macaulay. Due to Remark \ref{remcmin}, $J_{(G_1)_m}$ is Cohen-Macaulay and 
\begin{equation}\label{eq3.11}
\mathrm{depth}(A_1[x_m,y_m]/\mathrm{in}_{<}(J_{(G_1)_m}))=m+1.
\end{equation}
Since $\mathcal{S}/I$ is Cohen-Macaulay with $\mathrm{depth}(\mathcal{S}/I)=n+1$, we have $A_2/I_{2}^{''}$ is Cohen-Macaulay with 
\begin{equation}\label{eq3.12}
   \mathrm{depth}(A_2/I_{2}^{''})=(n+1)-(m+1)=n-m. 
\end{equation}
Similarly, we get $J_{(G_2)_m}$ is Cohen-Macaulay with
\begin{equation}\label{eq3.13}
\mathrm{depth}(A_2[x_m,y_m]/\mathrm{in}_{<}(J_{(G_2)_m}))=n-m+2,
\end{equation}
and $A_1/I_{1}^{''}$ is Cohen-Macaulay with 
\begin{equation}\label{eq3.14}
   \mathrm{depth}(A_1/I_{1}^{''})=(n+1)-(n-m+2)=m-1. 
\end{equation}
Now, by our given labelling, it follows that $(I_1:x_m)=\mathrm{in}_{<}(J_{(G_1)_m})+\big<y_{m+1}\big>$, $(\big<I_1,x_m\big>:y_m)=\big<I_{1}^{''},x_m\big>$ and $(I_1,x_m,y_m)=(J_1,x_m,y_m)$. Therefore, using equations \ref{eq3.10}, \ref{eq3.11}, \ref{eq3.14} and Remark \ref{remdepth}, we get the following:
\begin{align*}
    \mathrm{depth}(\mathcal{S}_{1}/(I_{1}:x_{m}))&=\mathrm{depth}(A_1[x_m,y_m]/\mathrm{in}_{<}(J_{(G_1)_m}))+1\\&=m+2;\\
 \mathrm{depth}(\mathcal{S}_{1}/(\big<I_{1},x_{m}\big>:y_{m}))&=\mathrm{depth}(A_1/I_{1}^{''})+3\\&=m+2;\\
 \mathrm{depth}(\mathcal{S}_{1}/\big<I_{1},x_{m},y_{m}\big>)&=\mathrm{depth}(A_1/J_1)+2\\&=m+2.
\end{align*}
By Lemma \ref{lemheuneke}, it follows that $\mathrm{depth}(\mathcal{S}_{1}/I_{1})=m+2$. By \cite[Lemma 3.13]{whisker}, $J_{\ov{G_{1}}}$ is unmixed, which gives $\mathrm{dim}(\mathcal{S}_{1}/I_{1})=m+2$ and so, $\mathcal{S}_{1}/I_{1}$ is Cohen-Macaulay. In a similar way, using equations \ref{eq3.10}, \ref{eq3.12}, \ref{eq3.13} and Remark \ref{remdepth}, we get 
\begin{align*}
    \mathrm{depth}(\mathcal{S}_{2}/(I_{2}:y_{m}))&=\mathrm{depth}(A_1[x_m,y_m]/\mathrm{in}_{<}(J_{(G_2)_m}))+1\\&=n-m+3;\\
    \mathrm{depth}(\mathcal{S}_{2}/(\big<I_{2},y_{m}\big>:x_{m}))&=\mathrm{depth}(A_2/I_{2}^{''})+3\\&=n-m+3=m+2;\\
    \mathrm{depth}(\mathcal{S}_{2}/\big<I_{2},x_{m},y_{m}\big>)&=\mathrm{depth}(A_2/J_2)+2\\&=n-m+3.
\end{align*}
Therefore, by Lemma \ref{lemheuneke}, $\mathrm{depth}(\mathcal{S}_{2}/I_{2})=n-m+3$. Since by \cite[Lemma 3.13]{whisker}, $J_{\ov{G_{2}}}$ is unmixed, $\mathrm{dim}(\mathcal{S}_{2}/I_{2})=n-m+3$ and so, $\mathcal{S}_{2}/I_{2}$ is Cohen-Macaulay. Hence, $J_{\ov{G_{1}}}$ and $J_{\ov{G_{2}}}$ are Cohen-Macaulay by Remark \ref{remcmin}.
\medskip

\noindent \textbf{Case II.} Let $m$ be a free vertex of at least one of $G_1$ and $G_2$. If $m$ is a free vertex of both $G_1$ and $G_2$, then the result holds by Remark \ref{remglu}. So, we may consider $m$ to be a free vertex of exactly one of $G_1$ and $G_2$. Without loss of generality, we assume $m$ is a free vertex of $G_{2}$ and not a free vertex of $G_{1}$. In this case, from our labelling of vertices of $G$ and looking at admissible paths in $G$, we get the following:
\begin{align*}
I:x_{m}&=\mathrm{in}_{<}(J_{(G_{1})_{m}})+I_{2}^{''};\\
\big<I,x_{m}\big>&=I_{1}^{'}+J_{2}+\big<x_{m}\big>.
\end{align*}
Since $J_{G}$ is Cohen-Macaulay, by Lemma \ref{lemonenf}, $A_{1}[y_m]/I_{1}^{'}$ and $A_2/J_2$ is Cohen-Macaulay. Again, $\mathcal{S}/I$ is Cohen-Macaulay implies $\mathcal{S}/(I:x_m)$ is Cohen-Macaulay by Lemma \ref{propcmcolon} and $\mathrm{depth}(\mathcal{S}/I)=\mathrm{depth}(\mathcal{S}/(I:x_m))$. Thus, $\mathrm{in}_{<}(J_{(G_1)_{m}})$ and $I_{2}^{''}$ are Cohen-Macaulay. Now, $\mathrm{depth}(\mathcal{S}/J_G)=n+1$, $\mathrm{depth}(A_{2}/J_{G_2\setminus\{m\}})=n-m+1$ and $\mathrm{depth}(J_{(G_1)_{m}})=m+1$. Therefore, using Remark \ref{remdepth} and Remark \ref{remcmin}, we get the following: 
\begin{align*}
\mathrm{depth}(A_{1}[x_{m},y_{m}]/\mathrm{in}_{<}(J_{(G_{1})_{m}}))&=m+1;\\
\mathrm{depth}(A_{2}/I_{2}^{''})&=(n+1)-(m+1)\\&=n-m;\\
\mathrm{depth}(A_{2}/J_{2})&=n-m+1;\\
\mathrm{depth}(A_{1}[y_{m}]/I_{1}^{'})&=(n+1)-(n-m+1)\\&=m.
\end{align*}

\noindent Hence, by simple calculation, we get $\mathrm{depth}(\mathcal{S}_{1}/(I_{1}:x_{m}))=m+1+1=m+2$; $\mathrm{depth}(\mathcal{S}_{1}/\big<I_{1},x_{m}\big>)=m+2$. Therefore, by Lemma \ref{lemheuneke}, $\mathrm{depth}(\mathcal{S}_{1}/I_{1})=m+2$. By \cite[Lemma 3.13]{whisker}, $J_{\ov{G_{1}}}$ is unmixed, which gives $\mathrm{dim}(\mathcal{S}_{1}/I_{1})=m+2$ and so, $\mathcal{S}_{1}/I_{1}$ is Cohen-Macaulay. Again, observe that 
$$\mathrm{depth}(\mathcal{S}_{2}/(I_{2}:x_{m}))=n-m+2+\mathrm{depth}(K[x_{m-1},y_{m}]/\big<x_{m-1}y_{m}\big>)=n-m+3,$$ and $\mathrm{depth}(\mathcal{S}_{2}/\big<I_{2},x_{m}\big>)=\mathrm{depth}(A_{2}/J_{2})+\mathrm{depth}(K[x_{m-1},y_{m-1},y_{m}]/\big<x_{m-1}y_{m}\big>)=n-m+3$. Thus, by Lemma \ref{lemheuneke}, $\mathrm{depth}(\mathcal{S}_{2}/I_{2})=n-m+3$. Since by \cite[Lemma 3.13]{whisker}, $J_{\ov{G_{2}}}$ is unmixed, $\mathrm{dim}(\mathcal{S}_{2}/I_{2})=n-m+3$ and so, $\mathcal{S}_{2}/I_{2}$ is Cohen-Macaulay. Therefore, $J_{\ov{G_{1}}}$ and $J_{\ov{G_{2}}}$ are Cohen-Macaulay by Remark \ref{remcmin}.
\end{proof}

\begin{definition}\label{defblockwhisker}{\rm
Let $G$ be a connected graph such that $J_{G}$ is unmixed and $B$ be a block of $G$. Let $V=\{v_{1},\ldots,v_{k}\}$ be the set of cut vertices of $G$ belonging to $V(B)$. Then, we can write
\begin{align}\label{blockdec}
G=B\cup\big(\bigcup_{i=1}^k G_{i}\big),
\end{align}
where $V(G_{i})\cap V(B)=\{v_{i}\}$, for each $1\leq i\leq k$, and the connected components of $G\setminus V$ are $B\setminus V$ (may be empty), $G_{1}\setminus\{v_{1}\},\ldots, G_{k}\setminus\{v_{k}\}$.

Considering the decomposition \ref{blockdec}, for $W=\{v_{s_{1}},\ldots,v_{s_{r}}\}\subseteq V$, we define a new graph $\overline{B}^{W}$ such that
\begin{enumerate}
\item[$\bullet$] $V(\overline{B}^{W})=V(B)\cup \big(\bigcup_{v_{i}\not\in W} V(G_{i})\big) \cup \{f_{v_{{s_{1}}}},\ldots, f_{v_{{s_{r}}}}\},$

\item[$\bullet$] $E(\overline{B}^W)=E(B)\cup \big(\bigcup_{v_{i}\not\in W} E(G_{i})\big)\cup \{\{v_{s_{i}},f_{v_{{s_{i}}}}\}\mid i=1,\ldots, r\}.$
\end{enumerate}
By $\overline{B}$ \textit{with respect to} $G$, we mean $\overline{B}^V$ 
and call it a \textit{block with whiskers} of $G$ (defined in \cite{s2acc}). 
Sometimes we simply write $\overline{B}$ if the graph is clear 
from the context. In simple words, $\overline{B}$ stands for 
that graph, which is obtained by attaching whiskers in $B$ to all the cut vertices $v_{i}$ of $G$ belonging  
to $V(B)$, replacing $G_{i}$'s. 
}
\end{definition}

As an application of Theorem \ref{thmGcm} and Theorem \ref{thmGbarcm}, we settle the open problem \cite[Question 5.11]{whisker} in the following corollaries.

 \begin{corollary}\label{corblockcm1}
Let $G$ be a graph such that $J_{G}$ is unmixed. If $J_{\overline{B}}$ is Cohen-Macaulay for each block $B$ of $G$, where $\overline{B}$ denotes the corresponding block with whiskers of $G$, then $J_{G}$ is Cohen-Macaulay under Hypothesis \ref{hypo}.
\end{corollary}
\begin{proof}
     We use induction on the number of vertices $n$ of $G$. For $n=1$ or $n=2$, the result follows trivially, and so, we assume $n>3$. If there exists a block $B$ of $G$ for which $\ov{B}=G$, then we are done. Suppose there is no such block of $G$. Then we can decompose $G$ as $G=G_{1}\cup G_{2}$ with $V(G_{1})\cap V(G_{2})=\{v\}$ such that $\ov{G_{i}}$ contains less than $n$ vertices, where $\ov{G_{i}}$ is obtained by attaching a whisker to $G_{i}$ at $v$ for $i=1,2$. By \cite[Lemma 3.13]{whisker}, we have $J_{\ov{G_{1}}}$ and $J_{\ov{G_{2}}}$ are unmixed. Note that each block of $\ov{G_{i}}$ is either a $K_{2}$ or a block of $G$. If $B$ is a block of $\ov{G_{i}}$ and $G$ both, then $\ov{B}$ with respect to both $\ov{G_{i}}$ and $G$ are same. Then by induction hypothesis, we get $J_{\ov{G_{1}}}$ and $J_{\ov{G_{2}}}$ are Cohen-Macaulay. Hence, $J_{G}$ is Cohen-Macaulay by Theorem \ref{thmGcm}.
\end{proof}
\begin{corollary}\label{corblockcm2}
Let $G$ be a graph. If $J_{G}$ is Cohen-Macaulay, then $J_{\overline{B}}$ is Cohen-Macaulay for each block $B$ of $G$, where $\overline{B}$ denotes the corresponding block with whiskers of $G$.
\end{corollary}
\begin{proof}
$J_{G}$ is Cohen-Macaulay implies $J_{G}$ is unmixed. We proceed by induction on the number of vertices $n$ of $G$. If $n=1$ or $n=2$, then $G$ is complete, and the result holds easily. Assume $n>2$. Let $B$ be a block of $G$ and $V=\{v_{1},\ldots,v_{k}\}$ be the set of cut vertices of $G$ belong to $V(B)$. If $\ov{B}$ is $G$, then the proof follows immediately. Let $\ov{B}$ is not same as $G$. Then there exists a $i\in\{1,\ldots,k\}$ for which $G_{i}$ is not a $K_{2}$. Consider the decomposition $G=G_{i}\cup B^{'}$, where $B^{'}=\ov{B}^{\{v_{i}\}}\setminus\{f_{v_{i}}\}$. Then $J_{\ov{B}^{\{v_{i}\}}}$ is Cohen-Macaulay by Theorem \ref{thmGbarcm}. Observe that $B$ is also a block in $\ov{B}^{\{v_{i}\}}$ and $\ov{B}$ with respect to $G$ is same as $\ov{B}$ with respect to $\ov{B}^{\{v_{i}\}}$. Therefore, $J_{\ov{B}}$ is Cohen-Macaulay by induction hypothesis.
\end{proof}

In \cite{whisker}, we have partially solved the open problem \cite[Problem 7.2]{acc}. Now, using Theorem \ref{thmGcm} and Theorem \ref{thmGbarcm}, we settle the open problem \cite[Problem 7.2]{acc} in the following Corollary \ref{coridentify} under Hypothesis \ref{hypo}. That is, it is enough to prove Hypothesis \ref{hypo} for \cite[Problem 7.2]{acc}, and thus, for the class of strongly unmixed binomial edge ideals, the following corollary holds.

\begin{corollary}\label{coridentify}
Let $G=G_{1}\cup G_{2}$, with $V(G_{1})\cap V(G_{2})=\{v\}$ and $H=H_{1}\cup H_{2}$, with $V(H_{1})\cap V(H_{2})=\{w\}$ be two distinct connected graphs, such that, $J_{G\setminus\{v\}},J_{H\setminus\{w\}}$ are unmixed. Consider the graph $F_{ij}=G_{i}\cup H_{j}$, 
identifying the vertices $v$ and $w$, labelling as $v$, i.e., 
$V(G_{i})\cap V(H_{j})=\{v\}$, for $i,j\in\{1,2\}$. Then, assuming Hypothesis \ref{hypo}, $J_{G}$ and $J_{H}$ are Cohen-Macaulay imply $J_{F_{ij}}$ is Cohen-Macaulay for every $i,j\in\{1,2\}$.
\end{corollary}
\begin{proof}
Consider the graphs $\ov{G_{i}}$ by attaching a whisker to $G_{i}$ at $v$ and $\ov{H_{i}}$ by attaching a whisker to $H_{i}$ at $w$ for $i=1,2$. Now, $J_{G}$ and $J_{H}$ are Cohen-Macaulay imply $J_{\ov{G_{1}}}$, $J_{\ov{G_{2}}}$, $J_{\ov{H_{1}}}$, $J_{\ov{H_{2}}}$ are Cohen-Macaulay by Theorem \ref{thmGbarcm}. From \cite[Theorem 3.19]{whisker}, we have $J_{F_{ij}}$ is unmixed for every $i,j\in \{1,2\}$. Hence, $J_{F_{ij}}$ is Cohen-Macaulay by Theorem \ref{thmGcm} for every $i,j\in\{1,2\}$.
\end{proof}

\section{Cohen-Macaulay Property with Girth of Graphs}\label{secgirth}
Let us first recall the definition of the girth of a graph.

\begin{definition}{\rm
The \textit{girth} of a simple graph $G$ is the length of a shortest induced cycle in $G$. If $G$ has no induced cycle, then the girth of $G$ is considered infinity. We denote the girth of $G$ by $\mathrm{girth}(G)$.
}
\end{definition}

There are lots of studies in Graph Theory on the girth of graphs, and it has many rich applications in several areas. For a planar graph, the girth of that graph is the edge connectivity of its dual graph, and vice versa \cite{thom83}. The girth of graphs is also important in the study of embedding of graphs. In particular, it is used to check the planarity of a graph using Euler's formula as we have $2m>gf$, where $m$ is the number of edges, $f$ is the number of faces, and $g$ is the girth of the graph. Hence, it would be an interesting study to connect the girth of graphs with some algebraic properties. We give a relation between the Cohen-Macaulay property of a binomial edge ideal and the girth of the corresponding graph in the following theorem.

\begin{theorem}\label{thmgirth}
Let $G$ be a graph. If $J_{G}$ is Cohen-Macaulay, then $\mathrm{girth}(G)\leq 4$ or $\mathrm{girth}(G)=\infty$.
\end{theorem}
\begin{proof}
We may assume $G$ is connected. Let the girth of $G$ be finite. Assuming $\mathrm{girth}(G)\geq 5$, we will prove $J_{G}$ is not Cohen-Macaulay using induction on $\mathrm{girth}(G)$. Suppose $\mathrm{girth}(G)=5$ and $G$ is accessible. Then there exists an induced $5$-cycle $H$ of $G$. Let $V(H)=\{v_{1},\ldots,v_{5}\}$ and $E(H)=\{\{v_{1},v_{2}\}, \{v_{2},v_{3}\}, \{v_{3},v_{4}\}, \{v_{4},v_{5}\}, \{v_{5},v_{1}\}\}$. Note that there is a minimal $T_{v_{1}}\subseteq \mathcal{N}_{G}(v_{1})$ such that after removing $T_{v_{1}}$ from $G$ there is no path from $v_{1}$ to $v_{3}$ or $v_{4}$. Since $T_{v_{1}}$ is minimal with such property, we have $T_{v_{1}}\in\mathscr{C}(G)$. It is obvious that $v_{2},v_{5}\in T_{v_{1}}$. Since $T_{v_{1}}$ is accessible, there exists $u\in T_{v_{1}}$ such that $T_{v_{1}}\setminus\{u\}\in \mathscr{C}(G)$. Since $J_{G}$ is unmixed, $c_{G}(T_{v_{1}}\setminus\{u\})=\vert T_{v_{1}}\vert$. Let $G_{1}^{v_{1}},\ldots, G_{\vert T_{v_{1}}\vert}^{v_{1}}$ be the connected components of $G\setminus (T_{v_{1}}\setminus \{u\})$ such that $v_{1},u\in V(G_{\vert T_{v_{1}}\vert}^{v_{1}})$. Again, there is a minimal $T_{u}\subseteq \mathcal{N}_{G}(u)$ such that after removing $T_{u}$ from $G$ there is no path from $u$ to $v_{4}$ (if $u=v_5$, then we will choose $T_{u}\subseteq \mathcal{N}_{G}(u)$ such that after removing $T_u$ from $G$ there is no path from $u$ to $v_3$). Then $T_{u}\in\mathscr{C}(G)$ and $v_{1}\in T_{u}$. Since $T_{u}$ is accessible, there exists $s\in T_{u}$ such that $T_{u}\setminus \{s\}\in \mathscr{C}(G)$ and $c_{G}(T_{u}\setminus\{s\})=\vert T_{u}\vert$. Let $G_{1}^{u},\ldots, G_{\vert T_{u}\vert}^{u}$ be the connected components of $G\setminus (T_{u}\setminus \{s\})$ such that $u,s\in V(G_{\vert T_{u}\vert}^{u})$. We consider the following two cases:
\medskip

\noindent \textbf{Case I.} Suppose $v_{1}=s$. If $z\in (T_{v_{1}}\setminus \{u\})\cap (T_{u}\setminus\{v_{1}\})$, then the induced subgraph of $G$ on $\{v_1,u,z\}$ is a cycle of length $3$, which is not possible as $\mathrm{girth}(G)=5$. Therefore, $(T_{v_{1}}\setminus \{u\})\cap (T_{u}\setminus\{v_{1}\})=\phi$ and $T= (T_{v_{1}}\setminus \{u\})\sqcup (T_{u}\setminus\{v_{1}\})\in \mathscr{C}(G)$. We can observe that the connected components in $G\setminus T$ are $G_{1}^{v_{1}},\ldots,G_{\vert T_{v_{1}}\vert -1}^{v_{1}}, G_{1}^{u},\ldots, G_{\vert T_{u}\vert -1}^{u}$ along with one connected component containing the vertices $v_{1},u$ and one connected component containing the vertices $v_{4}$. Therefore 
$$c_{G}(T)\geq \vert T_{v_{1}}\vert -1 +\vert T_{u}\vert -1 +2= \vert T_{v_{1}}\vert +\vert T_{u}\vert > \vert T\vert +1,$$
which gives a contradiction as $J_{G}$ is unmixed.
\medskip

\noindent \textbf{Case II.} Assume $v_{1}\neq s$. Consider the minimal $T_{s}\subseteq \mathcal{N}_{G}(s)$ such that after removing $T_{s}$ from $G$ there is no path from $s$ to $v_{5}$. Then $T_{s}\in\mathscr{C}(G)$ and $u\in T_{s}$. Since $T_{s}$ is accessible, there exists $w\in T_{s}$ such that $T_{s}\setminus \{w\}\in \mathscr{C}(G)$ and $c_{G}(T_{s}\setminus\{w\})=\vert T_{s}\vert$. Let $G_{1}^{s},\ldots, G_{\vert T_{s}\vert}^{s}$ be the connected components of $G\setminus (T_{s}\setminus \{w\})$ such that $s,w\in V(G_{\vert T_{s}\vert}^{s})$. We consider two sub-cases:
\medskip

\noindent \textbf{Case A.} Suppose $u\neq w$. If $z\in (T_{v_{1}}\setminus \{u\})\cap (T_{s}\setminus\{w\})$, then the induced subgraph of $G$ on $\{v_1,u,s,z\}$ is a cycle of length $4$, which is not possible as $\mathrm{girth}(G)=5$. Therefore, $(T_{v_{1}}\setminus \{u\})\cap (T_{s}\setminus\{w\})=\phi$ and it is easy to observe that $S= (T_{v_{1}}\setminus \{u\})\sqcup (T_{s}\setminus\{w\})\in \mathscr{C}(G)$. The connected components in $G\setminus S$ are at least $G_{1}^{v_{1}},\ldots,G_{\vert T_{v_{1}}\vert -1}^{v_{1}}, G_{1}^{s},\ldots, G_{\vert T_{s}\vert -1}^{s}$ along with one connected component containing the vertices $s,w$ and one connected component containing the vertices $v_{1}$. Therefore 
$$c_{G}(S)\geq \vert T_{v_{1}}\vert -1 +\vert T_{s}\vert -1 +2= \vert T_{v_{1}}\vert +\vert T_{s}\vert > \vert S\vert +1,$$
which gives a contradiction as $J_{G}$ is unmixed.
\medskip

\noindent \textbf{Case B.} Let $u=w$. Then $(T_{u}\setminus \{s\})\cap (T_{s}\setminus\{u\})=\phi$ and $L= (T_{u}\setminus \{s\})\sqcup (T_{s}\setminus\{u\})\in \mathscr{C}(G)$ as girth of $G$ is $5$. In $G\setminus L$, observe that $G_{1}^{u},\ldots,G_{\vert T_{u}\vert -1}^{u}, G_{1}^{s},\ldots, G_{\vert T_{s}\vert -1}^{s}$ are connected components along with one connected component containing the vertices $u,s$ and one connected component containing the vertex $v_{5}$. Therefore 
$$c_{G}(L)\geq \vert T_{u}\vert -1 +\vert T_{s}\vert -1 +2= \vert T_{u}\vert +\vert T_{s}\vert > \vert L\vert +1,$$
which gives a contradiction as $J_{G}$ is unmixed. Hence, our initial assumption was wrong, i.e., if $\mathrm{girth}(G)=5$, then $G$ can not be accessible and thus, $J_{G}$ can not be Cohen-Macaulay. 
\medskip

Let $\mathrm{girth}(G)=n>5$. Consider a vertex $v$ of an induced cycle of length $n$ in $G$. Then the girth of $G_{v}$ is exactly $n-1$. By induction hypothesis, $J_{G_{v}}$ is not Cohen-Macaulay. Hence, $J_{G}$ is not Cohen-Macaulay by Proposition \ref{propGv}.
\end{proof}

\begin{remark}
From the proof of Theorem \ref{thmgirth} and \cite[Lemma 4.5]{acc}, it follows that if $G$ is an accessible graph, then $\mathrm{girth}(G)\leq 4$ or $\mathrm{girth}(G)=\infty$
\end{remark}

The converse of Theorem \ref{thmgirth} is not true, and it is clear from the earlier work on Cohen-Macaulay binomial edge ideals. From Theorem \ref{thmgirth}, one can ask whether the bound on the girth of graphs can be sharper or not in the case of Cohen-Macaulay binomial edge ideals. The answer is no, and it is clear from \cite[Figure 3]{cactus}.

\section{Some Open Problems}\label{secprob}

We have proved Theorem \ref{thmGcm} under Hypothesis \ref{hypo}. Due to \cite[Conjecture 1.1]{acc} and \cite[Proposition 5.14]{acc}, it seems the hypothesis is true, but we need to prove it as the conjecture is not proven yet for general graphs. Also, suppose the conjecture is not true, then we can not say the hypothesis will not hold. Thus, Hypothesis \ref{hypo} is a weaker assumption than the conjecture. We strongly believe that Hypothesis \ref{hypo} is true, and so, we put the following question.

\begin{question}
    Prove or disprove Hypothesis \ref{hypo}.
\end{question}

Consider a graph $G$ as in Setup \ref{setup}. If $J_{G}$ is Cohen-Macaulay, then by Theorem \ref{thmGbarcm}
\begin{equation}\label{eqdepth}
\mathrm{depth}(\mathcal{S}/J_{G})=\mathrm{depth}(\mathcal{S}_{1}/J_{\ov{G_{1}}})+ \mathrm{depth}(\mathcal{S}_{2}/J_{\ov{G_{2}}})-4.
\end{equation}

If $G$ is decomposable into $G_{1}$ and $G_{2}$, then also the above equality (\ref{eqdepth}) holds by \cite[Theorem 2.7]{raufrin}. In general, the equality (\ref{eqdepth}) does not hold:

\begin{figure}[H]
	\centering
	\begin{tikzpicture}
  [scale=.6,auto=left,every node/.style={circle,scale=0.5}]
    
  \node[draw] (n1) at (0,4)  {$8$};
  \node[draw] (n2) at (0,0)  {$2$};
  \node[draw] (n3) at (4,4) {$6$};
   \node[draw] (n4) at (-2,2) {$4$};
   \node[draw] (n5) at (1,2) {$5$};
  \node[draw] (n6) at (4,0) {$3$};
  \node[draw] (n7) at (-3,4) {$9$};
  \node[draw] (n8) at (-3,0) {$1$};
  \node[draw] (n9) at (4,6.5) {$7$};
  \node[draw] (n10) at (-3,6.5) {$10$};
  \node[draw] (n11) at (0,6.5) {$11$};
  \node[draw] (n12) at (2,6.5) {$12$};

  \foreach \from/\to in {n1/n7,n1/n3, n1/n4, n2/n3, n2/n4, n2/n6, n2/n8, n3/n5, n3/n6, n3/n9, n4/n5, n5/n6, n7/n10, n10/n11, n1/n11, n11/n12}
    \draw[] (\from) -- (\to);
   
\end{tikzpicture}
\caption{Graph $G=G_{1}\cup G_{2}$ such that $V(G_{1})\cap V(G_{2})=\{8\}$}\label{figneqdepth}
\end{figure}
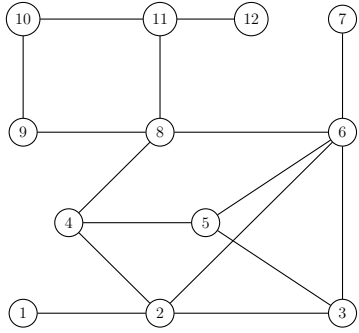

\begin{example}{\rm
Consider the graph $G$ in Figure \ref{figneqdepth} with the decomposition $G=G_{1}\cup G_{2}$, where $V(G_{1})=\{1,\ldots,8\}$ and $V(G_{2})=\{8,\ldots,12\}$. Using Macaulay2 \cite{mac2}, we see the equality (\ref{eqdepth}) does not hold for $G$ with this decomposition. In this case, we observe that $\mathcal{N}_{G_{1}}(8)\subseteq \{3,4,6\}\in \mathscr{C}(G_{1}\setminus\{8\})$ and $8$ is not a free vertex of $G_{2}$. 
}
\end{example}

\noindent Let $J_{\ov{G_{i}}}$ be Cohen-Macaulay for each $i=1,2$. If there exists $T_{1}\in \mathscr{C}(G_{1}\setminus \{v\})$ with $\mathcal{N}_{G_{1}}(v)\subseteq T_{1}$ and $v$ is not a free vertex of $G_{2}$, then $J_{G}$ is not unmixed (follows from the proof of Lemma \ref{lemG-vunm}) and hence, the equality (\ref{eqdepth}) does not hold in this case. Similarly, if we interchange $G_{1}$ and $G_{2}$, then we get same observation also. Considering all of these cases, we ask the following question.

\begin{question}{\rm
Consider a graph as in Setup \ref{setup}. Suppose $G$ satisfies the following:
\begin{enumerate}[(i)]
\item If there exists a cutset of $G_{1}\setminus \{v\}$ containing $\mathcal{N}_{G_{1}}(v)$, then $v$ is a free vertex of $G_{2}$;
\item If there exists a cutset of $G_{2}\setminus \{v\}$ containing $\mathcal{N}_{G_{2}}(v)$, then $v$ is a free vertex of $G_{1}$.
\end{enumerate}
Does the equality (\ref{eqdepth}) hold in this case?
}
\end{question}

For an ideal $I$ of $R$, we denote by $\mathrm{reg}(R/I)$, the Castelnuovo-Mumford regularity of $R/I$. Let $G$ be a graph as described in Setup \ref{setup}. If $v$ is a free vertex of both $G_1$ and $G_2$, then by \cite[Theorem 3.1]{jnr19}, we get $\mathrm{reg}(\mathcal{S}/J_{G})=\mathrm{reg}(\mathcal{S}_{1}/J_{\ov{G_{1}}})+ \mathrm{reg}(\mathcal{S}_{2}/J_{\ov{G_{2}}})-2$. Thus, keeping the equality (\ref{eqdepth}) in mind, we propose the following question.

\begin{question}{\rm
Let $G$ be a graph as described in Setup \ref{setup} and $v$ is not a free vertex of at least one of $G_1$ and $G_2$. Then, does there exist any relation among $\mathrm{reg}(\mathcal{S}/J_{G})$, $\mathrm{reg}(\mathcal{S}_{1}/J_{\ov{G_{1}}})$, $\mathrm{reg}(\mathcal{S}_{2}/J_{\ov{G_{2}}})$ under some suitable condition?
}
\end{question}

The next question is about the existence of some graphs with Cohen-Macaulay binomial edge ideals. From Theorem \ref{thmgirth}, we see that there exists no graph $G$ with Cohen-Macaulay $J_{G}$ having $4<\mathrm{girth}(G)<\infty$. In \cite[Example 3]{s2acc}, we see a graph having an induced cycle of length $5$ with Cohen-Macaulay binomial edge ideal, but in this case, the girth of the graph is $3$. The next question is about graphs with girth $4$.

\begin{question}{\rm
Does a graph $G$ of girth $4$ with Cohen-Macaulay $J_{G}$ exist such that $G$ has an induced cycle of length $5$ or more? Till now, we have not seen any accessible graph having an induced cycle of length $6$ or more. Does there exist such a graph?
}
\end{question}

\bibliographystyle{amsalpha}

\end{document}